\documentclass[a4paper,11pt]{article}
\usepackage{amsmath, amssymb, amsthm,amsfonts,mathabx,bm}
\usepackage[english]{babel}
\usepackage{bbm}
\usepackage{graphicx}
\usepackage{url}
\usepackage{epstopdf}
\usepackage[a4paper,bindingoffset=0.5cm,left=2cm,right=2cm,top=2.5cm,bottom=2cm,footskip=.8cm]{geometry}

\usepackage{tikz}
\usetikzlibrary{arrows, automata,positioning,calc,shapes,decorations.pathreplacing,decorations.markings,shapes.misc,petri,topaths}
  \usepackage{pgfplots}
  \pgfplotsset{compat=newest}
  \usetikzlibrary{plotmarks}
  \usepackage{grffile}
\newlength\figureheight
  \newlength\figurewidth
\pgfplotsset{%
    tick label style={font=\scriptsize},
    label style={font=\footnotesize},
    legend style={font=\footnotesize},
         every axis plot/.append style={very thick}
}

\usepackage{rotating}
\usepackage{amsbsy,enumerate}
\usepackage{graphicx}
\usepackage{ccaption}
\usepackage{comment}
\usepackage{mathrsfs}

\newcommand{\ee}{\mathbb{E}}

\newcommand{\pp}{\mathbb{P}}

\allowdisplaybreaks

\usepackage{amssymb,amsmath,color}

\newcommand\eq[1]{(\ref{eq:#1})}
\newcommand\cadlag{c\`adl\`ag\ }
\newtheorem{theorem}{Theorem}
\newtheorem{lemma}{Lemma}
\newtheorem{corollary}{Corollary}
\newtheorem{remark}{Remark}
\newtheorem{condition}{Condition}

\newcommand{\blue}[1]{{\color{blue}#1}}

\begin{document}

\title{From reflected L\'evy processes to stochastically monotone Markov processes via generalized inverses and supermodularity}

\author{Offer Kella\thanks{Department of Statistics and Data Science, the Hebrew University of Jerusalem, Jerusalem 9190501, Israel ({\tt\scriptsize  offer.kella@huji.ac.il}).}
\thanks{Supported in part by grant No. 1647/17 from the Israel Science Foundation and the Vigevani Chair in Statistics.}
\ \ and\ \  Michel Mandjes\thanks{Korteweg-de Vries Institute for Mathematics, University of Amsterdam, Science Park 904, 1098 XH Amsterdam, The Netherlands ({\tt\scriptsize  m.r.h.mandjes@uva.nl}).}
\thanks{Partly funded by the NWO Gravitation Programme N{\sc etworks} (Grant Number 024.002.003) and an NWO Top Grant (Grant Number 613.001.352).}}
	
\maketitle

\begin{abstract} \noindent
It was recently proven that the correlation function of the stationary version of a reflected L\'evy process  is nonnegative, nonincreasing and convex. In another branch of the literature it was established that the mean value of the reflected process starting from zero is nonnegative, nondecreasing and concave. In the present paper it is shown, by putting them in a common framework, that these results extend to  substantially more general settings. 
Indeed, instead of reflected L\'evy processes, we consider a class of more general stochastically monotone Markov processes. In this setup we  show monotonicity results associated with a supermodular function of two coordinates of our Markov process,  from which the above-mentioned monotonicity and convexity/concavity  results directly follow, but now for the class of Markov processes considered rather than just reflected L\'evy processes. In addition, various results for the transient case (when the Markov process is not in stationarity) are provided. 
The conditions imposed are natural, in  that they are satisfied by various frequently used Markovian models as illustrated by a series of examples.

\vspace{2mm}
 
 
 \vspace{2mm}
 
 \noindent 
{\it AMS Subject Classification (MSC2010):} 60J99, 60G51, 90B05.

\end{abstract}
 
 \section{Introduction}
 
In the context of {L\'evy-driven queues} \cite{DM} and  L\'evy storage processes \cite{KY},
it was recently shown \cite{B19} that, whenever the stationary distribution exists and has a finite second moment,  the correlation function associated with the stationary version of the reflected process is nonnegative, nonincreasing and convex. Here, a L\'evy-driven queue is to be interpreted as the one-sided (Skorokhod) reflection map applied to a L\'evy process. Notably, the results in \cite{B19} show that the mentioned structural properties carry over to the finite-buffer L\'evy-driven queue, {\em i.e.,} the two-sided (Skorokhod) reflection map. One could regard \cite{B19} as the endpoint of a long-lasting research effort which began over four decades ago. The nonnegativity, nonincreasingness and convexity of the correlation function of the stationary process was proven in \cite{OTT} for the case where the L\'evy process under consideration is compound Poisson. The more recent contributions
\cite{EM} and \cite{GM}  deal with the spectrally-positive and negative cases, respectively. Finally, \cite{B19} removed the spectral restrictions on the L\'evy process assumed in \cite{EM,GM}. Whereas \cite{EM,GM,OTT} rely on the machinery of completely-monotone functions, \cite{B19} uses a direct conditioning argument in combination of elementary properties of the reflection map.

A second strand of research that we would like to mention concerns structural properties of the mean value (and related quantities) of the reflected process. It was found \cite{KS94} that for a one-sided Skorohod reflection, when the driving process has stationary increments and starts from zero, the mean  of the reflected process (as a function of time) is nonnegative, nondecreasing and concave. In particular this holds when the driving process also has independent increments ({\em i.e.,} the L\'evy case), which for the spectrally-positive case had been discovered earlier  \cite{K92}, where we refer to \cite[Thm.\ 11]{KW96}
for a multivariate analogue. The nonnegativity, nondecreasingness and concavity of the mean was proven to extend to the two-sided reflection case in \cite{AM}, where it was also shown that for the one- and two-sided reflection cases the variance is nondecreasing. 
 
The main objective of this paper is to explore to what level of generality the results from the above two branches of the literature can be extended, and whether they could be somehow brought under a common umbrella. 
Importantly, in our attempt to understand the above-mentioned structural properties better, we discovered that they are covered by  a substantially broader framework. We have done so by considering stochastically monotone Markov processes (in both discrete and continuous time), and deriving properties of the expected value of bivariate supermodular functions of coordinates of the process. 

Importantly, we discovered a neat and quite simple approach to extend a broad range of existing results to a substantially broader class of processes and more general functional setups. {We strongly feel that this particular approach gets to the heart of the matter, and also helps in giving a much clearer understanding of earlier results.} More specifically, our findings directly 
imply the type of monotonicity results of the covariance that were found in \cite{B19,EM,GM} and \cite{K92,KS94}. For the convexity results of \cite{B19,EM,GM} and the concavity results of \cite{KS94} (restricted to L\'evy processes) and \cite{K92} a further, rather natural, condition needs to be imposed on the underlying Markov transition kernel~-- this is Condition~\ref{cond:1} to follow. However, notably, the monotonicity of the variance established in \cite{AM} is {\it not}  valid under the conditions imposed in the current paper, as we will show by means of a counterexample. 
 
The area of stochastically monotone Markov processes is vast. Without aiming at giving a full overview, we would like to mention \cite{Daley,KK77,rusch,siegmund}. In particular, in \cite{Daley} a main result is Theorem~4, stating that if $\{X_n\,|\,n\ge 0\}$ is a  stationary stochastically monotone time-homogeneous Markov chain (on a real valued state space) and $f$ is nondecreasing, then $\text{Cov}(f(X_0),f(X_n))$ (whenever it exists and is finite) is nonnegative and nonincreasing in $n$. As it turns out, this result as well is a special case the results established in our current paper. Importantly, quite a few frequently used stochastic processes are stochastically monotone Markov processes, covering for example birth-death processes and diffusions  \cite{KK77}, as well as certain L\'evy dams, (state dependent) random walks, besides the above-mentioned reflected processes.

In our proofs we use the notion of a generalized inverse of a distribution function and some of its properties, conditioning arguments and the application of the concept of supermodularity and its relationship to comonotonicity.
More concretely, it will be important to study the properties of $h(X_s,X_t)$ or $h(X_s,X_t-X_{t+\delta})$ (and others) for $0\le s\le t$ and $\delta>0$, where $h$ is a supermodular function. This will be done for both the stationary case and the transient case (under various conditions). For background on results associated with supermodular functions, and in particular the relationship with comonotone random variables, which we will need several times, we refer to \cite{CSW76,PS10,pw15}.

As so often in mathematics, once being in the right framework proofs can be highly compact and seemingly straightforward. It is, however, typically far from trivial to 
identify this best `lens' through which one should look at the problem. This phenomenon also applies in the context of the properties derived in the present paper. Indeed, in previous works the focus has been on specific models and specific properties, with proofs that tend to be ad-hoc, lengthy and involved, reflecting the lack of an overarching framework. 
With the general approach that we develop in this paper, we manage to bring a wide class of existing results under a common denominator, with the underlying proofs becoming clean and insightful. In addition, because we have found the right angle to study this class of problems, we succeed in shedding light on the question to what extent these results can be further generalized. Our objective was to present the framework as cleanly as possible; our paper is self-contained in the sense that it does not require any previous knowledge of stochastically monotone Markov processes.

The paper is organized as follows. 
In Section~\ref{main} the formal setup (including Condition~\ref{cond:1}), the main results and their proofs will be given. In Section~\ref{Motivation} we provide a series of appealing examples of frequently used stochastically monotone Markov processes that satisfy Condition~\ref{cond:1}. Section \ref{CONCL} concludes.

Throughout we write $a\wedge b=\min(a,b)$, $a\vee b=\max(a,b)$, $a^+=a\vee 0$, $a^-=-a\wedge 0=(-a)^+$. In addition, {\em a.s.} abbreviates {\em almost surely} ({\em i.e.,}  {\em with probability one}), and {\em cdf} abbreviates {\em cumulative distribution function}.
Also, $=_{\rm d}$ abbreviates `distributed' or `distributed like' and $X\le_{\text{st}}Y$ means ${\mathbb P}(X>t)\le {\mathbb P}(Y>t)$ for all $t$ (stochastic order).

\section{General theory}\label{main}
This section presents our main results. Section \ref{MR} treats our general theory, whereas Section \ref{SUP} further reflects on general questions that can be dealt with relying on  our results, as well as the connection with supermodular functions. 

\subsection{Main results}\label{MR}
For $x\in\mathbb{R}$ and $A$ Borel (one-dimensional), we let $p(x,A)$ be a Markov transition kernel. By this we mean that for every Borel $A$, $p(\cdot,A)$ is a Borel function and for each $x\in\mathbb{R}$, $p(x,\cdot)$ is a probability measure. We will say that $p$ is {\em stochastically monotone} if $p(x,(y,\infty))$ is nondecreasing in $x$ for each $y\in\mathbb{R}$, which, as discussed earlier and will be demonstrated later, is a natural property across a broad range of frequently used stochastic models.

The following condition plays a crucial role in our results. Whenever it is satisfied, it allows us to establish highly general results. The condition is natural in the context of, {\em e.g.,}  queues and other storage systems, as will be pointed out in Section \ref{Motivation}.

\begin{condition}
\label{cond:1}$p(x,(x+y,\infty))$ is nonincreasing in $x$ for each $y\in\mathbb{R}$. 
\end{condition}

Now, for $n\ge 1$, let $p$ and $p_n$, for $n\ge 1$, be transition kernels. Define
\begin{equation}
G(x,u)=\inf\{y\,|\, p(x,(-\infty,y])\ge u\}
\end{equation}
the {\it generalized-inverse function} associated with the cdf $F_x(y)=p(x,(-\infty,y])$, and let similarly $G_n(x,u)$ be the generalized-inverse function associated with $p_n$. We recall ({\em e.g.,} \cite{eh13}, among many others) that $G(x,u)$ is nondecreasing and left continuous in $u$ (on $(0,1)$), and that $G(x,u)\le y$ if and only if $u\le p(x,(-\infty,y])$. Thus, if $U=_{\rm d}\text{U}(0,1)$, where $\text{U}(0,1)$ is the uniform distribution on $(0,1)$, then ${\mathbb P}(G(x,U)\le y)=p(x,(-\infty,y])$ and thus ${\mathbb P}(G(x,U)\in A)=p(x,A)$. A similar reasoning applies to $G_n(x,u)$ for every $n\ge 1$.

\begin{lemma}\label{lem:g}
$p$ is stochastically monotone if and only if, for each $u\in(0,1)$, $G(x,u)$ is nondecreasing in $x$. Furthermore, Condition~\ref{cond:1} is satisfied if and only if, for each $u$, $G(x,u)-x$ is nonincreasing in $x$. 
\end{lemma}
\begin{proof}
Follows from the facts~(i) that for $x_1<x_2$
\begin{equation}
\{y\,|\,p(x_1,(-\infty,y])\ge u\}\supset \{y\,|\,p(x_2,(-\infty,y])\ge u\}\,,
\end{equation}
(ii)~that under Condition~\ref{cond:1} for $x_1<x_2$
\begin{equation}
\{y\,|\,p(x_1,(-\infty,x_1+y])\ge u\}\subset \{y\,|\,p(x_2,(-\infty,x_2+y])\ge u\},
\end{equation}
and (iii) that $G(x,u)-x=\inf\{y\,|\,p(x,(-\infty,x+y]\ge u\}$.
\end{proof}

Now, denote $g_k^{k}(x,u)=G_k(x,u)$ and, for $n\ge k+1$,
\begin{equation}
g_k^n(x,u_1,\ldots,u_{n-k+1})=G_n(g_k^{n-1}(x,u_1,\ldots,u_{n-k}),u_{n-k+1})\ .
\end{equation}
It immediately follows by induction that in case $p_k,\ldots,p_n$ are stochastically monotone, it holds that $g_k^n(x,u_1,\ldots,u_{n-k+1})$ is nondecreasing in $x$.
Assuming that $U_1,U_2,\ldots$ are i.i.d.\ and distributed $\text{U}(0,1)$, then with $X'_0=x$ and 
\begin{equation}
X'_n=g_1^n(x,U_1,\ldots,U_n)
\end{equation} 
for $n\ge 1$, $\{X'_n\,|\, n\ge 0\}$ is a real valued (possibly time-inhomogenous) Markov chain with possibly time-dependent transition kernels $p_1,p_2,\ldots$.
Let us now denote $p^{k-1}_k(x,A)=1_A(x)$ and, for $n\ge k$,
\begin{equation}
p^n_k(x,A)=\int_{\mathbb{R}}p_n(y,A)p_k^{n-1}(x,\text{d}y)\,.
\end{equation}

\begin{lemma}\label{lem:n-step}
If, for $1\le k\le n$, $p_k,\ldots,p_n$ are stochastically monotone Markov kernels (resp., and in addition satisfy Condition~\ref{cond:1}), then $p_k^n$ is stochastically monotone (resp., and in addition satisfies Condition~\ref{cond:1}). 
\end{lemma}
\begin{proof}
By induction, it suffices to show this for the case $n=k+1$. If $p_k$ and $p_{k+1}$ are stochastically monotone, then $g_k^{k+1}(x,U_1,U_2)$ is a random variable having the distribution $p_k^{k+1}(x,\cdot)$. Therefore, the stochastic monotonicity of $p_k^{k+1}$ is a consequence of the fact that $g_k^{k+1}(x,U_1,U_2)$ is nondecreasing in $x$. Now, if stochastic monotonicity and Condition~\ref{cond:1} hold then $G_k(x,U_1)-x$ and $G_{k+1}(G_k(x,U_1),U_2)-G_k(x,U_1)$ are nonincreasing in $x$ and thus so is their sum. This implies that $g_k^{k+1}(x,U_1,U_2)-x$ is nonincreasing in $x$, which implies that $p_k^{k+1}$ satisfies Condition~\ref{cond:1}.
\end{proof}
A (possibly time-inhomogeneous) Markov chain with stochastically monotone transition kernels will be called a {\em stochastically monotone Markov chain} (see, {\it e.g.}, \cite{Daley} for the time-homogenous case). 
Lemma~\ref{lem:n-step} immediately implies the following.

\begin{corollary}\label{cor:sub}
Any subsequence of a stochastically monotone Markov chain (resp.,  in addition satisfying Condition~\ref{cond:1}) is also a stochastically monotone Markov chain (resp., in addition satisfying Condition~\ref{cond:1}).
\end{corollary}
Therefore, a subsequence of a time-homogeneous stochastically monotone Markov chain (resp., in addition satisfying Condition~\ref{cond:1}) may no longer be time-homogeneous, but is always a stochastically monotone Markov chain (resp., in addition satisfying Condition~\ref{cond:1}).

Recall that $h:\mathbb{R}^2\to\mathbb{R}$ is called {\em supermodular} if whenever $x_1\le x_2$ and $y_1\le y_2$ we have that
\begin{equation}
h(x_1,y_2)+h(x_2,y_1)\le h(x_1,y_1)+h(x_2,y_2)\ .
\end{equation}
If $X$ and $Y$ have cdf\,s $F_X$ and $F_Y$, then $(X,Y)$ will be called {\em comonotone} if $\pp(X\le x,Y\le y)=\pp(X\le x)\wedge \pp(Y\le y)$ for all $x,y$. There are various equivalent definitions for comonotonicity. In particular it is worth mentioning that when $X$ and $Y$ are identically distributed, then they are comonotone if and only if $\pp(X=Y)=1$. It is well known that if  $(X',Y')$ is comonotone and has the same marginals as $(X,Y)$, then for any Borel supermodular $h$ for which $\ee h(X,Y)$ and $\ee h(X',Y')$ exist and are finite, we have that $\ee h(X,Y)\le \ee h(X',Y')$. In particular, when $X,Y$ are identically distributed then $\ee h(X,Y)\le \ee h(Y,Y)$, which is a property that we will need later in this paper. For such results and much more see, {\em e.g.}, \cite{pw15} and references therein, where the Borel assumption was missing, but is actually needed as there are non-Borel supermodular functions for which $h(X,Y)$ is not necessarily a random variable. We write down what we will need later as a lemma. Everything in this lemma is well known.

\begin{lemma}\label{lem:supermod}
Let $(X,Y)$ be a random pair such that $X=_{\rm d} Y$. Then for every Borel supermodular function $h:\mathbb{R}^2\to\mathbb{R}$ for which $\ee h(X,Y)$ and $\ee h(Y,Y)$ exist and are finite, we have that
\begin{equation}
\ee h(X,Y)\le \ee h(Y,Y)\ .
\end{equation}  
Moreover, if $h$ is supermodular and $f_1,f_2$ are nondecreasing, then $h(f_1(x),f_2(y))$ is supermodular and in particular, since $h(x,y)=xy$ is supermodular, $f_1(x)f_2(y)$ is supermodular as well.
\end{lemma}

As usual, we call $\pi$ {\em invariant} for a Markov kernel $p$ if, for every Borel $A$,  $\int_{\mathbb{R}}p(x,A)\pi(\text{d}x)=\pi(A)$.
We proceed by stating and proving our first main result.

\begin{theorem}\label{th:stat012}
Assume that $X_0,X_1,X_2$ is a stochastically monotone Markov chain where $p_1$ has an invariant distribution $\pi_1$ and $X_0$ is $\pi_1$ distributed. Then for every Borel supermodular $h:\mathbb{R}^2\to\mathbb{R}$,
\begin{equation}\label{eq:stat012}
\ee h(X_0,X_2)\le \ee h(X_1,X_2)
\end{equation}
whenever the means exist and are finite. In particular, for any nondecreasing $f_1,f_2$ for which the means of $f_1(X_0)$, $f_2(X_2)$, $f_1(X_0)f_2(X_2)$ and $f_1(X_1)f_2(X_2)$ exist and are finite, we have that
\begin{equation}\label{eq:cov012}
0\le\text{\rm Cov}(f_1(X_0),f_2(X_2))\le \text{\rm Cov}(f_1(X_1),f_2(X_2))\ . 
\end{equation}
\end{theorem}

\begin{proof}
Let $X'_0,U_1,U_2$ be independent with $U_1,U_2=_{\rm d}\text{U}(0,1)$ and $X'_0= X_0$. Then with $X'_1=G_1(X'_0,U_1)$ and $X'_2=G_2(X'_1,U_2)$ we have that $(X'_0,X'_1,X'_2)=_{\rm d} (X_0,X_1,X_2)$. Now, we note that since  $G_2(y,u_2)$ is nondecreasing in $y$, then (due to Lemma~\ref{lem:supermod}) $h(x,G_2(y,u_2))$ is supermodular in $x,y$ for every fixed $u_2$. Since $X'_0=_{\rm d} X'_1$, then it follows from Lemma~\ref{lem:supermod} that
\begin{equation}
\ee (h(X'_0,G_2(X'_1,U_2))\,|\,U_2)\le \ee (h(X'_1,G_2(X'_1,U_2))\,|\,U_2)\,.
\end{equation}
Taking expected values on both sides gives \eq{stat012}. Noting that $\ee f_1(X_0)\ee f_2(X_2)=\ee f_1(X_1)\ee f_2(X_2)$ and that $f_1(x)f_2(y)$ is supermodular gives the right inequality in \eq{cov012}. To show the left inequality, note that $G_2(G_1(x,u_1),u_2)$ is a nondecreasing function of $x$ and thus, so is $\gamma(x)=\ee f_2(G_2(G_1(x,U_1,U_2)))$. Now,
\begin{equation}
\text{Cov}(f_1(X_0),f_2(X_2))=\text{Cov}(f_1(X_0),\ee [f_2(X_2)|X_0])=\text{Cov}(f_1(X_0),\gamma(X_0))
\end{equation}
and it is well known that the covariance of comonotone random variables (whenever well defined and finite) must be nonnegative.
\end{proof}

A time-homogenous continuous-time Markov process $\{X_t\,|\,t\ge 0\}$ will be called stochastically monotone, whenever $p^t(x,A)=\pp_x(X_t\in A)$ is a stochastically monotone kernel for each $t>0$. We will say that it satisfies Condition~\ref{cond:1} whenever $p^t$ satisfies this condition for every $t>0$. Note that by Corollary~\ref{cor:sub} this is equivalent to the assumption that these conditions are satisfied for $0<t\le \epsilon$ for some $\epsilon>0$.

\begin{theorem}\label{cor:stat012hom}
Consider a stationary stochastically monotone discrete-time or continuous-time time-homogenous Markov process $\{X_t\,|\,t\ge0\}$. Then for every supermodular $h$ for which the following expectations exist and are finite, $\ee h(X_0,X_t)$ is nonincreasing in $t\ge 0$ where $t$ is either nonnegative integer or nonnegative real valued. In particular, when $f_1,f_2$ are nondecreasing and the appropriate expectations exist, $\text{\rm Cov}(f_1(X_0),f_2(X_t))$ is nonnegative and nonincreasing in $t\ge0$.
\end{theorem}
\begin{proof}
For every $0<t_1<t_2$ we have that $X_0,X_{t_2-t_1},X_{t_2}$ satisfy the conditions and hence the conclusions of Theorem~\ref{th:stat012} (for the discrete time case, recall Corollary~\ref{cor:sub}). By stationarity we have that $(X_{t_2-t_1},X_{t_2})=_{\rm d} (X_0,X_{t_1})$. Therefore
\begin{equation}
\ee h(X_0,X_{t_2})\le \ee h(X_{t_2-t_1},X_{t_2})=\ee h(X_0,X_{t_1})\,.
\end{equation}
Note that, since $X_t=_{\rm d} X_0$, Lemma~\ref{lem:supermod} implies that
\begin{equation}
\ee h(X_0,X_t)\le \ee h(X_0,X_0)
\end{equation}
so that $\ee h(X_0,X_t)$ is nonincreasing on $[0,\infty)$ and not just on $(0,\infty)$.
Since $\ee f_1(X_0)\ee f_2(X_t)=\ee f_1(X_0)\ee f_2(X_0)$, the result for the covariance follows by taking $h(x,y)=f_1(x)f_2(y)$.
\end{proof}

\begin{remark}\label{rem:mixing}{\rm
The following is a standard and very simple exercise in ergodic theory. Let $T$ be a measure-preserving transformation on $(\Omega,\mathcal{F},\mu)$, where $\mu$ is a $\sigma$-finite measure. This means that $\mu(T^{-1}(A))=\mu(A)$ for every $A\in\mathcal{F}$, where $T^{-1}(A)=\{\omega\in \Omega\,|\,T(\omega)\in A\}$. Then $T$ is mixing (in the sense that $\mu(A\cap T^{-n}B)\to \mu(A)\mu(B)$ as $n\to\infty$, for every $A,B\in\mathcal{F}$) if and only if for every $f_1,f_2:\Omega\to\mathbb{R}$ such that $\int f_i^2\text{d}\mu<\infty$ for $i=1,2$ we have that $\int f_1\cdot T^nf_2\text{d}\mu \to \int f_1\text{d}\mu\cdot  \int f_2\text{d}\mu$, where $T^nf_2(\omega)=f_2(T^n\omega)$. 

The implication of this, under the assumptions of Theorem~\ref{cor:stat012hom}, is that with such mixing conditions $\text{Cov}(f_1(X_0),f_2(X_n))\to 0$ for every Borel $f_1,f_2$ such that $\ee f_i^2(X_0)<\infty$ for $i=1,2$. This in particular holds when in addition $f_i$, $i=1,2$, are nondecreasing. This also implies that the same would hold in the continuous-time case as the covariance is nonincreasing (when $f_i$ are nondecreasing) and thus it suffices that it vanishes along any subsequence (such as $t_n=n$). In particular this will hold for any Harris-recurrent Markov process for which there exists a stationary distribution. In this case an even stronger form of mixing is known to hold (called {\em strong mixing} or $\alpha$-mixing, {\em e.g.}, see \cite{AP86}). All the examples discussed in Section~\ref{Motivation} for which a stationary distribution exists are in fact Harris-recurrent and even have a natural regenerative state. Alternatively, the same holds whenever the stationary distribution is unique and $f_1,f_2$ are nondecreasing with $\ee f_i(X_0)^2<\infty$ for $i=1,2$. This may be shown by adapting the proof of Theorem~4 of \cite{Daley} in which $f_1=f_2$. For all of our examples in which a stationary distribution exists, it is also unique.

Of course, we cannot expect that the covariance will vanish without such mixing conditions or uniqueness of the stationary distribution. For example, if $\xi$ is some random variable having a finite second moment and variance $\sigma^2>0$, set $X_t=\xi$ for all $t\ge 0$. Then $\{X_t\,|\, t\ge 0\}$ is trivially a stochastically monotone, stationary Markov process (and also trivially satisfies Condition~\ref{cond:1}), but (taking $f_1(x)=f_2(x)=x$) $\text{Cov}(X_0,X_t)=\sigma^2$ clearly does not vanish as $t\to\infty$. \hfill$\diamond$
}\end{remark}

We continue with two theorems in which Condition~\ref{cond:1} is imposed.
\begin{theorem}\label{th:stat0123}
Assume that $X_0,X_1,X_2,X_3$ is a stochastically monotone Markov chain satisfying Condition~\ref{cond:1}, where $p_1$ has an invariant distribution $\pi_1$ and $X_0$ is $\pi_1$ distributed. Then for every Borel supermodular $h:\mathbb{R}^2\to\mathbb{R}$,
\begin{equation}\label{eq:stat0123}
\ee h(X_0,X_2-X_3)\le \ee h(X_1, X_2-X_3)
\end{equation}
whenever the expectations exist and are finite. In particular, for $f$ nondecreasing and the appropriate expectations exist and are finite, then
\begin{equation}\label{eq:cov0123}
0\le \text{\rm Cov}(f(X_0),X_2)-\text{\rm Cov}(f(X_0),X_3)\le \text{\rm Cov}(f(X_1),X_2)-\text{\rm Cov}(f(X_1),X_3)\,.
\end{equation}
\end{theorem}

\begin{proof}
The proof is very similar to the proof of Theorem~\ref{th:stat012}. That is, we let $X'_0=X_0$, $X'_n=G_n(X'_{n-1},U_n)$ for $n=1,2,3$ where $X'_0,U_1,U_2,U_3$ are independent and $U_1,U_2,U_3=_{\rm d} \text{U}(0,1)$. From the stochastic monotonicity and Condition~\ref{cond:1} it follows that $G_2(x,u_2)-G_3(G_2(x,u_2),u_3)$ is nondecreasing in $x$. Therefore, by Lemma~\ref{lem:supermod} we have that, since $h(x,G_2(y,u_2)-G_3(G_2(y,u_2),u_3))$ is supermodular in $x,y$ and $X'_1=_{\rm d} X'_0$, we have that
\begin{align}
\ee [h(X'_0,X'_2-X'_3)|U_2,U_3]&=\ee [h(X'_0,G_2(X'_1,U_2)-G_3(G_2(X'_1,U_2),U_3))|U_2,U_3]\nonumber\\
&\le \ee [h(X'_1,G_2(X'_1,U_2)-G_3(G_2(X'_1,U_2),U_3))|U_2,U_3]\\
&=\ee [h(X'_1,X'_2-X'_3)|U_2,U_3]\ ,\nonumber
\end{align}
and taking expected values establishes \eq{stat0123}. Taking $h(x,y)=f(x)y$ gives the right inequality of \eq{cov0123}. The left inequality is obtained via comonotonicity, by observing that since $G_2(G_1(x,u_1),u_2)-G_3(G_2(G_1(x,u_1),u_2),u_3)$ is nondecreasing in $x$, then
$\ee [X_2-X_3|X_0]$ is a nondecreasing function of $X_0$, so that this inequality follows from the comonotonicity of $f(X_0)$ and $\ee [X_2-X_3|X_0]$ as in the proof of the left inequality of \eq{cov012}.
\end{proof}

\begin{theorem}\label{cor:stat0123hom}
Consider a stationary stochastically monotone discrete-time or continuous-time time-homogenous Markov process $\{X_t\,|\,t\ge0\}$, satisfying Condition~\ref{cond:1}. Then for every $s>0$ and every supermodular $h$ for which the following expectations exist and are finite, $\ee h(X_0,X_t-X_{t+s})$ is nonincreasing in $t$ where $t$ is either nonnegative integer or nonnegative real valued. In particular, when $f$ is nondecreasing and the appropriate expectations exist, $\text{\rm Cov}(f(X_0),X_t)$ is nonnegative, nonincreasing and convex in $t$.
\end{theorem}
In particular, note that when choosing $f(x)=x$ and assuming that $\ee X_0^2<\infty$, we see that, under the conditions of Theorem~\ref{cor:stat0123hom}, the auto-covariance $R(t)=\text{Cov}(X_s,X_{s+t})$ (or auto-correlation $R(t)/R(0)$ when $X_0$ is not a.s. constant) is nonnegative, nonincreasing and convex in $t$. 

\begin{proof}
Let $t_1<t_2$ then $X_0,\, X_{t_2-t_1},\, X_{t_2-t_1+s},\, X_{t_2+s}$ satisfy the conditions and hence the conclusion of Theorem~\ref{th:stat0123}. Therefore,
\begin{equation}
\ee h(X_0,X_{t_2}-X_{t_2+s})\le \ee h(X_{t_2-t_1},X_{t_2}-X_{t_2+s})=\ee h(X_0,X_{t_1}-X_{t_1+s})
\end{equation}
where the right equality follows from stationarity. When $f$ is nondecreasing then $h(x,y)=f(x)y$ is supermodular and thus
$\ee f(X_0)X_{t+s}-\ee f(X_0) X_t$ is nonincreasing in $t$ for every $s>0$. Therefore $\ee f(X_0)X_t$ is midpoint convex and since by Theorem~\ref{cor:stat012hom} it is nonnegative and nonincreasing (hence Borel), it must be convex (see \cite{Blumberg,Sierpinski}).
\end{proof}

Can anything be said for the case where the initial distribution is not invariant? Here is one possible answer.
\begin{theorem}\label{thm:sm}
Let $\{X_t\,|\,t\ge0\}$ be a  stochastically monotone discrete-time or continuous-time time-homogenous Markov process. Assume that the initial distribution can be chosen so that $X_0\le X_t$ a.s.\ for every $t\ge 0$. Then, 
\begin{description}
\item{(i)} $X_t$ is stochastically increasing in $t$.
\item{(ii)} For every Borel supermodular function which is nondecreasing in its first variable and for which the expectations exist and are finite, $\ee h(X_s,X_t)$ is nondecreasing in $s$ on $[0,t]$. When in addition Condition~\ref{cond:1} is satisfied, the same is true for $\ee h(X_s,X_t-X_{t+\delta})$ for every $\delta>0$ (whenever the expectations exist and are finite). 
\item{(iii)} When $h$ is nondecreasing in both variables (not necessarily supermodular) and expected values exist and are finite, then $\ee h(X_s,X_{t+s})$ is nondecreasing in $s$. When Condition~\ref{cond:1} is satisfied, the same is true for
$\ee h(X_s,X_{s+t}-X_{s+t+\delta})$ for $\delta>0$.
\item{(iv)} When $\ee X_t$ exists and is finite then it is nondecreasing and under Condition~\ref{cond:1} it is also concave.
\end{description}
\end{theorem}

We note that it would suffice to assume that $X_0\le X_t$ for $t\in(0,\epsilon]$ for some $\epsilon>0$, or in discrete time for $t=1$. Also, we note that for (i) we can replace $X_0\le X_t$ by $X_0\le_{\text{st}}X_t$.

\begin{proof}
For any $s<t$ take $\epsilon\in(0,t-s]$ and let $G_v(x,u)$ be the generalized-inverse with respect to the kernel $p^v$. Let $U_0,U_1,\ldots$ be i.i.d. with $U_i=_{\rm d}\text{U}(0,1)$. Since $X_0\le X_\epsilon$ we have with $X'_0=X_0$, that $X'_0\le G_\epsilon(X'_0,U_0)$ and thus
\begin{equation}
X'_s\equiv G_s(X'_0,U_1)\le G_s(G_\epsilon(X'_0,U_0),U_1)\equiv X''_{s+\epsilon}=_{\rm d} X'_{s+\epsilon}\equiv G_\epsilon(X'_s,U_2)\,,
\end{equation}
implying stochastic monotonicity. 

Taking $X'_t=G_{t-s-\epsilon}(X'_{s+\epsilon},U_3)$, then $(X'_0,X'_s,X'_{s+\epsilon},X'_t)$ is distributed like $(X_0,X_s,X_{s+\epsilon},X_t)$.
Since $h$ is nondecreasing in its first variable and $X'_s\le X''_{s+\epsilon}$, it follows that
\begin{equation}\label{eq:h}
\ee h(X'_s,X'_t)\le \ee h(X''_{s+\epsilon},X'_t)\ .
\end{equation}
The (by now, repetitive) fact that
\begin{align}
\ee [h(X''_{s+\epsilon},X'_t)|U_3]&=\ee[h(X''_{s+\epsilon},G_{t-s-\epsilon}(X'_{s+\epsilon},U_3)|U_3]\nonumber \\ &\le \ee[h(X'_{s+\epsilon},G_{t-s-\epsilon}(X'_{s+\epsilon},U_3)|U_3]=\ee[h(X'_{s+\epsilon},X'_t)|U_3]
\end{align}
follows from the supermodularity of $h(x,G_{t-s-\epsilon}(y,u_3))$ in $x,y$. Taking expected values implies, together with \eq{h}, that $\ee h(X_s,X_t)$ is nondecreasing in $s$ on $[0,t]$.  The proof of the fact that, under Condition~\ref{cond:1}, $\ee h(X_s,X_t-X_{t+\delta})$ is nondecreasing in $s$ on $[0,t]$ is similar, once we define $X'_{t+\delta}=G_\delta(X'_t,U_4)$ and observe that $X'_t-X'_{t+\delta}=G_{t-s-\epsilon}(X'_{s+\epsilon},U_3)-G_\delta(G_{t-s-\epsilon}(X'_{s+\epsilon},U_3),U_4)$ is nondecreasing in $X'_{s+\epsilon}$.

When $h$ is nondecreasing in both variables we have that $h(X_s,X_{s+t})=_{\rm d} h(X_s,G_t(X_s,U_0))$ so that by stochastic monotonicity $\ee (h(X_s,G_t(X_s,U_0))|U_0)$ is nondecreasing in $s$ and hence also $\ee h(X_s,X_{s+t})$. The proof for
$\ee h(X_s,X_{s+t}-X_{s+t+\delta})$, under Condition~\ref{cond:1}, is similar.

Finally, since $X_t$ is stochastically increasing then it clearly follows that $\ee X_t$ is nondecreasing. When Condition~\ref{cond:1} is met, then taking $h(x,y)=y$ (nondecreasing in both variables) we have that $\ee X_{s+t}-\ee X_{s+t+\delta}$ is nondecreasing. This implies midpoint concavity, so that since $\ee X_t$ is monotone (hence Borel) it follows that it is concave (again, see \cite{Blumberg,Sierpinski}).
\end{proof}

We complete this section by noting that although, for the sake of convenience, all the results were written for the case where the state space is $\mathbb{R}$, they hold whenever the state space is any Borel subset of $\mathbb{R}$ as was assumed in \cite{Daley}; for instance, for  the examples discussed in Section~\ref{Motivation}, the state space is either $[0,\infty)$ or $\mathbb{Z}_+$.

{
\subsection{Role of stochastic monotone Markov processes and supermodularity}\label{SUP}
To put our work into perspective, we conclude this section by explaining the train of thought that led us to this research. One result that sparked this line of research was that for a stationary reflected (both one- and two-sided) L\'evy process $X$  with $\ee X_0^2<\infty$,
$R(t)=\text{Cov}(X_0,X_t)$ is nonnegative, nonincreasing and convex \cite{B19, EM,GM}. Since in this case $\ee X_t=\ee X_0$, this is equivalent to $\ee X_0X_t$ being nonnegative, nonincreasing and convex. Another result is that for a stationary discrete time stochastically monotone Markov process and a nondecreasing function $g$ with $\ee g^2(X_0)<\infty$, $\text{Cov}(g(X_0),g(X_n))$  is nonincreasing  \cite{Daley}. This is equivalent to $\ee g(X_0)g(X_n)$ being nonincreasing.  We observed that $xy$, $g(x)g(y)$ (when $g$ is nondecreasing) are supermodular functions of $x,y$.
Also, we knew that the one- and two-sided Skorohod reflection of a process of the form $x+X_t$ is monotone in $x$. When $X$ is a L\'evy process, then the reflected process is Markovian and from the monotonicity in $x$ it must be stochastically monotone.

Therefore, a natural question arose: is it true that for any stationary stochastically monotone Markov process (discrete or continuous time) and for any supermodular $h$ we have that $\ee h(X_0,X_t)$ is nonincreasing? If yes, then this would make the just mentioned results from \cite{B19,Daley,EM,GM}, each using a totally different (often quite involved) approach, simple special cases. We were happy to see that the answer to this question was `yes'.

The next step was to try to identify what condition ensures convexity of the correlation or even more general function. As seen in Section~\ref{main} the condition that ensures convexity of $\ee g(X_0)X_t$ for a nondecreasing function $g$ (in particular $g(x)=x$) was Condition~\ref{cond:1}. We lack a full intuitive understanding as to why it is this precise condition that makes it work. 
As will be seen in Section~\ref{Motivation}, various natural Markovian models satisfy stochastic monotonicity and Condition~\ref{cond:1}.

Finally, we also knew that for a (one- and two-sided) reflected L\'evy process $X$ starting from the origin (and actually also more general processes), $\ee X_t$ is nonincreasing and concave \cite{AM,K92,KS94,KW96}. We were asking ourselves whether this too may be seen as a special case of the more general theory we discovered, and, as it turned out, the answer to this question was affirmative as well.
}

 \section{Examples}\label{Motivation}

The purpose of this section is to give a number of motivating examples for the kind of processes to which our theory applies.
We chose these examples to illustrate the huge potential of the methodology developed in this paper: it can serve to make proofs of existing results significantly more transparent and compact (the examples with reflected L\'evy processes), and it can serve to derive entirely new structural properties (the examples with the L\'evy dam, the state dependent random walk and the birth and death processes).

\subsection{L\'evy process reflected at the origin} \label{1reflect}\label{rlp}
Consider a \cadlag L\'evy process $Y=\{Y_t\,|\,t\ge 0\}$ with $\pp(Y_0=0)=1$ (not necessarily spectrally one-sided).  For every $x$, the one-sided (Skorokhod) reflection map, with reflection taking place at level 0, is defined through
\begin{equation}
X_t(x)=x+Y_t-\inf_{0\le s\le t}(x+Y_s)\wedge0=Y_t+L_t\wedge x
\end{equation}
where $L_t=-\inf_{0\le s\le t}Y_s\wedge 0$ with $L_t(x)=(L_t-x)^+$ (so, in particular, $L_t=L_t(0)$). The pair $(L_t(x),X_t(x))$ is known to be the unique process satisfying
\begin{description}
\item{(i)} $L_t(x)$ is right continuous, nondecreasing in $t$, with $L_0=0$.
\item{(ii)} $X_t(x)$ is nonnegative for every $t\ge 0$.
\item{(iii)} For every $t>0$ such that $L_s(x)<L_t(x)$ for every $s<t$ we have that $X_t(x)=0$.
\end{description}
It is known  \cite{kella06} that (iii) is equivalent to the condition that \begin{equation}\int_{[0,\infty)}X_s(x)L_{\text{d}s}(x)=0,\end{equation}
or alternatively to the condition that $L_t(x)$ is the minimal process satisfying (i) and (ii). Special cases of such processes are the workload process in an M/G/1 queue and the (one dimensional) reflected  Brownian motion (where the reflection takes place at $0$).

It turns out that when $x=0$, then $\ee X_t$ is nondecreasing and concave in $t$ \cite{K92,KS94}. When $\ee Y_1<0$, then this process, which is well known to be Markovian, has a stationary distribution. If $W$ has this stationary distribution and is independent of the process $Y$, then $X^*_t=X_t(W)$ is a stationary process and it has been established \cite{EM,GM,OTT} that when $\ee W^2<\infty$, the autocovariance $R(t)=\text{Cov}(X_s^*,X_{s+t}^*)$ (or autocorrelation $R(t)/R(0)$) is nonnegative, nonincreasing and convex in~$t$. Where earlier proofs tended to be ad-hoc (e.g.\ dealing with spectrally one-sided processes only, {\em i.e.,} assuming that the L\'evy process has jumps in one direction) and involved (e.g.\ requiring delicate manipulations with completely monotone functions), with the techniques developed in the present paper this property now follows virtually immediately. The remainder of this subsection illustrates this.

Let us denote $p^t(x,A)=\pp(X_t(x)\in A)$, so that $p^t(\cdot,\cdot)$ is the transition kernel of the associated Markov process. 
Since $X_t(x)=Y_t+L_t\wedge x$ is nondecreasing in $x$ and $X_t(x)-x=Y_t-(x-L_t)^+$ is nonincreasing in $x$ (for each $t>0$) we clearly have that
\begin{equation}
p^t(x,(y,\infty))=\pp(X_t(x)>y)
\end{equation} 
is nondecreasing in $x$ and
\begin{equation}
p^t(x,(x+y,\infty))=\pp(X_t(x)>x+y)=\pp(Y_t-(x-L_t)^+>y)
\end{equation}
is nonincreasing in $x$. Therefore stochastic monotonicity and Condition~\ref{cond:1} are satisfied.

\begin{remark}\label{rem:GG1}{\rm
We note that for the same reasons, a (general) random walk reflected a the origin is a discrete time version of the process featuring in the above setup. As a consequence, it is also stochastically monotone and satisfies Condition~\ref{cond:1}. In particular, this applies to the consecutive waiting times upon arrivals of customers in a GI/GI/1 queue.}
\end{remark}

\subsection{L\'evy process with a two-sided reflection}\label{drlp}
In this subsection we argue that the structural properties discussed in the previous subsection carry over to the case of a two-sided reflection.
With $Y$ defined in Section~\ref{1reflect}, a two-sided (Skorokhod) reflection in $[0,b]$ for $b>0$ (and similarly in $[a,b]$ for any $a<b$) is defined as the unique process $(X_t(x),L_t(x),U_t(x))$, with $X_t(x)=x+Y_t+L_t(x)-U_t(x)$, satisfying
\begin{description}
\item(i) $L_t(x),\,U_t(x)$ are right continuous and nondecreasing with $L_0(x)=U_0(x)=0$.
\item(ii) $X_t(x)\in [0,b]$ for all $t\ge 0$.
\item(iii) For every $t>0$ such that $L_s(x)<L_t(x)$ (resp., $U_s(x)<U_t(x)$) for every $s<t$, $X_t(x)=0$ (resp., $X_t(x)=b$).
\end{description}
Also here, (iii) is equivalent to \begin{equation}\int_{[0,\infty)}X_t(x)L_{\text{d}t}(x)=\int_{[0,\infty)}(b-X_t(x))U_{\text{d}t}(x)=0.\end{equation} 

{
For this case it is also known that $\ee X_t(0)$ is nondecreasing and concave \cite{AM} as well as that for the stationary version the autocovariance $R(t)$ is nonnegative nondecreasing and convex \cite{B19}. As in the one-sided case, we can apply our newly developed results to establish these facts in a very compact manner provided we verify that the process under consideration is a stochastically monotone Markov process that fulfills Condition~\ref{cond:1}.
}

Since $Y$ is a L\'evy process, we have that $X_t(x)$ is a time-homogenous Markov process starting at $x$. The driving process $Y$ being the same for both $X_t(x)$ and $X_t(y)$, we find that choosing $x<y$ means that $X_t(x)$ can never overtake $X_t(y)$. Consequently, $X_t(x)$ is nondecreasing in $x$ and thus the Markov chain is stochastically monotone (with some effort, this can also be shown directly from representation \eq{double} to follow). 

In order to verify that it satisfies Condition~\ref{cond:1}, we recall from \cite{klrs07}, upon re-denoting by $X^0_t(x)$ the one-sided reflected process described in Section~\ref{1reflect}, that
\begin{equation}\label{eq:double}
X_t(x)=X^0_t(x)-\sup_{0\le s\le t}\left[\left(X_s^0(x)-b\right)^+\wedge\inf_{s\le u\le t}X_t^0(x)\right]\ .
\end{equation}
Since $X_t^0(x)$ is nondecreasing and $X_t^0(x)-x$ is nonincreasing in $x$ (as explained in Section~\ref{1reflect}), it immediately follows that $X_t(x)-x$ is nonincreasing in $x$, which implies Condition~\ref{cond:1}. 

{
Regarding the results for $R(t)$ that hold under stationarity, observe that
in this two-sided reflected case a stationary distribution always exists and has a bounded support. Therefore, we do not need to impose any additional requirements on $Y$. This is in contrast to the one-sided case where it was needed to assume that $\ee Y_1<0$ and that the stationary distribution has a finite second moment. As in the case with one-sided reflection, the findings carry over to the discrete time counterpart; the two-sided reflected random walk.
}

\subsection{L\'evy dams with a nondecreasing  release rule}
In the previous two subsections, dealing 
with the one and two-sided reflected L\'evy process, we mentioned that existing (and also some previously nonexistent) results may be instantly concluded from the theory that we develop in this paper.
The present subsection gives an illustration of our theory's potential to conclude similar results for dam processes which, to the best of our knowledge, is completely new for these kind of processes and was not known earlier. More concretely, it shows that, with the general theory that we developed, the
structural results discussed above carry over to more than just reflected L\'evy processes.

Let the process $J=\{J_t\,|\,t\ge 0\}$ be a right continuous subordinator (nondecreasing L\'evy process) with $\pp(J_0=0)=1$ and let $r:[0,\infty)\to [0,\infty)$ be nondecreasing, left continuous on $(0,\infty)$, with $r(0)=0$. Consider the following {\it dam process}:
\begin{equation}\label{eq:dam}
X_t(x)=x+J_t-\int_0^tr(X_s(x))\text{d}s\,.
\end{equation}
It is well known \cite{cp72} that, under the stated assumptions, the solution to \eq{dam} is unique (pathwise) and belongs to the class of time-homogeneous Markov processes. 

As before, we need to check that the process under consideration is stochastically monotone and fulfils Condition~\ref{cond:1}. For $x<y$ we have that
\begin{equation}
X_t(y)-X_t(x)=y-x-\int_0^t(r(X_s(y)-r(X_s(x))\text{d}s\ .
\end{equation}
Denote $\tau$ to be first time (if it exists) for which the right side is zero. Because $x<y$, then for every $t<\tau$ we have that $X_t(x)<X_t(y)$. On $\tau<\infty$ we clearly  have that $X_\tau(x)=X_\tau(y)$ and therefore we also have that for any $h\ge 0$,
\begin{align}
X_{\tau+h}(x)&=X_\tau(x)+J_{\tau+h}-J_\tau+\int_0^h r(X_{\tau+s}(x))\text{d}s\nonumber\\
&=X_\tau(y)+J_{\tau+h}-J_\tau+\int_0^h r(X_{\tau+s}(x))\text{d}s=X_{\tau+h}(y)\,,
\end{align}
where the right equality follows from the uniqueness of the solution $Z$ to the equation
\begin{equation}
Z_h=z+J_{\tau+h}-J_\tau-\int_0^hr(Z_s)\text{d}s\,.
\end{equation}
Therefore, we have that $X_t(x)\le X_t(y)$ for every $t\ge 0$. Moreover, note that
\begin{equation}
(X_t(x)-x)-(Y_t(y)-y)=\int_0^t(r(X_s(y))-r(X_s(x))\text{d}s
\end{equation}
and thus, since $r$ is assumed to be nondecreasing, we also have that $X_t(x)-x\ge X_t(y)-y$. The conclusion is that $X_t(x)$ is nondecreasing in $x$ and $X_t(x)-x$ is nonincreasing in $x$. In other words, the process considered is a stochastically monotone Markov process satisfying Condition~\ref{cond:1}. 
As a consequence, the structural properties on $\ee X_t$ and $R(t)$, as we discussed above, apply to this process as well. 
Regarding $R(t)$, we note that here a stationary distribution exists whenever $\ee J_1<r(x)$ for some $x>0$ (recalling that $r(\cdot)$ is nondecreasing).

So as to perform a sanity check, we note that when choosing $r(x)=rx$ the resulting process is  a (generalized) shot-noise process. In this case it is well known that we can explicitly write 
\begin{equation}X_t(x)=xe^{-rt}+\int_{(0,t]}e^{-r(t-s)}J_{\text{d}s}.
\end{equation} In this setting it is immediately clear that $X_t(x)$ is nondecreasing and $X_t(x)-x$ is nonincreasing in $x$. We observe that here, if $\ee X_0^2<\infty$, then $R(t)=\text{Cov}(X_0,X_t)=\text{Var}(X_0)e^{-rt}$, so that $R(t)/R(0)=e^{-rt}$, which is, as expected, nonnegative, nonincreasing, convex in $t$ (for any distribution of $X_0$ having a finite second moment) and also converges to zero as $t\to\infty$. It is well known that in this particular case, the stationary distribution has a finite second moment if and only if $J_1$ has a finite second moment. This  is equivalent to requiring that $\int_{(1,\infty)}x^2\nu(\text{d}x)<\infty$, where $\nu$ is the associated L\'evy measure.

\begin{remark}{\rm
We note that if we replace $J$ by any L\'evy process and $-r(\cdot)$ by $\mu(\cdot)$ where $\mu$ is nonincreasing (not necessarily negative), then whenever the process described by \eq{dam} is well defined, we have for the same reasons as for the dam process, stochastic monotonicity and the validity of  Condition~\ref{cond:1}. This is in particular the case when $J_t=\sigma B_t$ and $B$ is a Wiener process,  resulting in a diffusion with constant diffusion coefficient and a nonincreasing drift.}
\end{remark}

\subsection{Discrete time (state dependent) random walks}\label{RW}
Consider a discrete time random walk on $\mathbb{Z}_+$ with the following transition probabilities:
\begin{equation}
p_{ij}=\begin{cases}
q_i&i\ge 1,\,j=i-1\\
r_i&i\ge 0,\,j=i\\
p_i&i\ge 0,\,j=i+1\\
0&\text{otherwise,}
\end{cases}
\end{equation}
where $q_i+r_i+p_i=1$ for $i\ge 1$ and $r_0+p_0=1$. It is a trivial exercise to check that $p(i,A)=\sum_{j\in A}p_{ij}$ is stochastically monotone if and only if $p_{i-1}\le 1-q_i$ for $i\ge 1$ and satisfies Condition~\ref{cond:1} if and only if $q_i$ is nondecreasing in $i$ and $p_i$ is nonincreasing in $i$. From Lemma~\ref{lem:n-step} it may be concluded that these are also the respective conditions for this Markov chain to be stochastically monotone and satisfies Condition~\ref{cond:1}. Therefore, the condition that both are satisfied is that for each $i\ge 1$
\begin{equation}
p_i\le p_{i-1}\le 1-q_i\le 1-q_{i-1}
\end{equation}
where $q_0\equiv 0$. 

From this it immediately follows that if $r_i=0$ for all $i\ge 1$ then the stochastic monotonicity together with Condition~\ref{cond:1} is equivalent to the assumption that $p_i=p$ for $i\ge 0$ and $r_0=q_i=1-p$ for $i\ge 1$. This is precisely the reflected process associated with $\sum_{k=1}^n Y_k$ (zero for $n=0$) where $Y_k$ are i.i.d. with ${\mathbb P}(Y=1)=1-{\mathbb P}(Y=-1)=p$ which is a simple special case of Remark~\ref{rem:GG1}.

\subsection{Birth and death processes}
Consider a right-continuous birth and death process on $\mathbb{Z}_+$ with birth rates $\lambda_i$, $i\ge 0$ and death rates $\mu_i\ge 1$. If the process is explosive we assume that the `cemetary' state is $+\infty$ and that from a possible time of explosion the process remains there forever (a minimal construction, that is).

 As observed in \cite{KK77}, this Markov process is stochastically monotone. To see this (for the sake of ease of reference), simply start two such independent processes from different initial states. By independence they do not jump at the same time (a.s.) and thus as long as they do not meet in a given state one remains strictly above the other. If and when they do meet (stopping time + strong Markov property) we let them continue together. If the higher process explodes then it remains at infinity forever and thus the bottom process never overtakes the top one. This coupling immediately implies stochastic monotonicity.
 
 Therefore it remains to identify when Condition~\ref{cond:1} holds. For the random walk of Subsection~\ref{RW} the condition was that $p_i$ is nonincreasing and $q_i$ is nondecreasing in $i$. Therefore it is conceivable that the condition here should be that $\lambda_i$ is nonincreasing and $\mu_i$ is nondecreasing in $i$. 
 This indeed turns out to be correct. In order to see this we will first need the following lemma.
 
\begin{lemma}\label{LBD}
$\Rightarrow 2$: For $k=1,2$, let $X^k=\{X_t^k|\,t\ge 0\}$ be a right continuous birth and death process on $\mathbb{Z}_+$ with birth rates $\lambda_i^k$ for $i\ge 0$ and death rates $\mu_i^k$ for $i\ge 1$. Assume that when the process is explosive then the `cemetary' is $+\infty$ and that the process remains there from the possible point of explosion and on. Then, the following two conditions are equivalent:
\begin{description}
\item{(i)}  $X^1_t\le_{{\rm st}}X^2_t$ for every $t\ge 0$ for any choice of initial distributions such that $X^1_0\le_{{\rm st}} X^2_0$.
\item{(ii)} $\lambda_i^1\le \lambda_i^2$ for every $i\ge 0$ and $\mu_i^1\ge \mu_i^2$ for every $i\ge 0$.
\end{description}
\end{lemma} 

\begin{proof}
We use the standard notation ${\mathbb P}_i$ to indicate the probability measure when the initial state is $i$. Also recall that a right continuous birth and death process is {\em orderly} in the sense that the probably that there are two or more jumps in the interval $[0,t]$ starting from any initial state is $o(t)$. This implies that $\lim_{t\downarrow 0}t^{-1}{\mathbb P}_i(|X^k_t-i|\ge 2)=0$ and thus 
\begin{equation}
\lim_{t\downarrow0}\frac{1}{t}{\mathbb P}_i(X^k_t\ge i+1)=\lim_{t\downarrow0}\frac{1}{t}{\mathbb P}_i(X_t^k=i+1)=\lambda_i
\end{equation} 
for $i\ge 0$. Similarly, $\lim_{t\downarrow0}t^{-1}{\mathbb P}_i(X^k_t\le i-1)=\mu_i$ for $i\ge 1$.

To show (i) $\Rightarrow$ (ii), suppose that $X_0^1=X_0^2=i$ a.s.\ (clearly $X^1_0\le_{\text{st}} X^2_0$). From $X^1_t\le_{\text{st}}X^2_t$ we have that ${\mathbb P}_i(X_t^1\ge i+1)\le {\mathbb P}_i(X_t^2\ge i+1)$ for $i\ge 0$ and $t>0$. Dividing by $t$ and letting $t\downarrow 0$ gives $\lambda^1_i\le \lambda^2_i$. Similarly, since ${\mathbb P}_i(X_t^1\ge i)\le {\mathbb P}_i(X_t^2\ge i)$ we have that
${\mathbb P}_i(X_t^1\le i-1)\le {\mathbb P}_i(X_t^2\ge i-1)$ for $i\ge 1$ and $t>0$. Dividing by $t$ and letting $t\downarrow 0$ gives $\mu_i^1\ge \mu_i^2$.

It remains to show (ii) $\Rightarrow$ (i). If we can find a coupling such that $X^1_t\le X^2_t$ a.s.\ for each $t\ge 0$ then the Lemma instantly follows. It is well known that if $X_0^1\le_{\text{st}}X_0^2$, then there are $(\tilde X^1_0,\tilde X^2_0)$ such that $\tilde X^k_0=_{\rm d} X^k_0$ and $\tilde X^1_0\le \tilde X^2_0$ a.s. Now assume that $(X^1,X^2)$ is a right continuous Markov process on $\{(i,j)|\,0\le i\le j\}$ with the following transition rates (with $\mu^1_0=\mu^2_0=0$).
For $i<j$ we have
\begin{equation}
\begin{array}{ccccc}
&&(i,j+1)\\ \\
&&{\scriptstyle\lambda^2_j}\Big\uparrow\hspace{3mm}\\ \\
(i-1,j)&\stackrel{\mu^1_i}{\longleftarrow}&(i,j)&\stackrel{\lambda^1_i}{\longrightarrow}&(i+1,j)\\ \\
&&{\scriptstyle\mu^2_j}\Big\downarrow\hspace{3mm}\\ \\
&&(i,j-1)
\end{array}
\end{equation}

\medskip\noindent
whereas for $i=j$
\begin{equation}
\begin{array}{ccccc}
&&(i,i+1)&&(i+1,i+1)\\ \\
&&{\scriptstyle\lambda^2_i-\lambda^1_i}\Big\uparrow\hspace{8mm}&{\scriptstyle\lambda^1_i}\nearrow\\ \\
(i-1,i)&\stackrel{\mu^1_i-\mu^2_i}{\longleftarrow}&(i,i)\\ \\
&{\scriptstyle\mu^2_i}\swarrow\\ \\
(i-1,i-1)
\end{array}
\end{equation}
All other transition rates are zero. Since the process lives on $\{(i,j)|\,0\le i\le j\}$ it is clear that $X^1_t\le X^2_t$ a.s.\ for each $t$. It is easy to check that for each $k=1,2$, $X^k$ is a birth a death process with birth rates $\lambda^k_i$ and death rates $\mu^k_i$.
\end{proof}

Note that the processes suggested in the proof run independently as long as they do not meet, and only if and when they meet something else happens that ensures that the bottom process does not overtake the top one.

\begin{theorem}\label{CBD}
Assume that $X=\{X_t|\,t\ge 0\}$ is a right continuous birth and death process with birth rates $\lambda_i$ for $i\ge 0$ and death rates $\mu_i$ for $i\ge 1$. Also assume that the `cemetery' (when explosive) is $+\infty$ and that the process remains there forever from such an epoch and on. Then $X$ satisfies Condition~1 if and only if $\lambda_i$ is nonincreasing in $i$ and $\mu_i$ is nondecreasing in $i$.
\end{theorem}

\begin{proof}
Let us us define $X(i)$ a birth and death process with the above parameters which starts from $i$. We need to show that $X_t(i+1)-(i+1)\le_{\text{st}} X_t(i)-i$, or, equivalently, that
$X_t(i+1)\le_{\text{st}} X_t(i)+1$. Now $X_t(i)+1$ is a birth and death process on $\mathbb{Z}_+$ starting from $i+1$ with arrival rates $\tilde \lambda_i=\lambda_{i-1}$ for $i\ge 1$ and $\tilde \lambda_0$ may be chosen arbitrarily in $[\lambda_0,\infty)$. The departure rates of $X_t(i)+1$ are $\tilde\mu_i=\mu_{i-1}$ for $i\ge 2$ and $\tilde\mu_1=0$.

From Lemma~\ref{LBD} it follows that the necessary and sufficient condition that we seek is $\tilde \lambda_i\ge \lambda_i$ for $i\ge 0$ and $\tilde\mu_i\le \mu_i$ which is equivalent to the condition that $\lambda_i$ is nonincreasing and $\mu_i$ is nondecreasing (in $i$).
\end{proof}

\begin{remark}{\rm
Note that a pure birth or pure death process is a special case of birth and death processes and thus of Theorem~\ref{LBD} and Corollary~\ref{CBD} hold for these processes as well. Also note that the state space does not need to be $\mathbb{Z}_+$, but can also be any finite or infinite subinterval of $\mathbb{Z}$. }
\end{remark}

\section{Some concluding remarks}\label{CONCL}
In this paper have focused on techniques to prove structural properties of Markov processes. We have developed a framework that succeeds in bringing various branches of the literature under a common umbrella. 
As indicated earlier, we strongly believe that the results presented in this paper open up the opportunity to produce compact proofs of existing results, as well as to efficiently derive new properties. We conclude this paper with a couple of remarks.

Recalling Remark~\ref{rem:mixing}, when there exists a stationary distribution, for all of the examples discussed in Section~\ref{Motivation}  we have that $\text{Cov}(f_1(X_0),f_2(X_t))$ vanishes as $t\to\infty$ whenever $f_1,f_2$ are nondecreasing with $\ee f_i^2(X_0)<\infty$ for $i=1,2$. 
We note that in light of the findings of Sections~\ref{rlp} and~\ref{drlp},  \cite[Thms.~1 and~2]{B19} (as well as the earlier \cite[Thm.\ 3.1]{EM} and \cite[Thm.\ 2.2]{GM}) are special cases of our Theorems~\ref{cor:stat012hom} and~\ref{cor:stat0123hom}. In addition, \cite[Thm.~3.1]{KS94} (which holds for any reflected process with stationary, not necessarily independent, increments, hence not necessarily Markov) restricted to the L\'evy case (and the earlier  \cite[Thm.\ 3.3]{K92}) as well as the mean (not variance) parts of  \cite[Thms.~4.6 and~7.5]{AM} are special cases of (iv) of our Theorem~\ref{thm:sm}  (upon taking $X_0=0$ in \cite{AM}, that is).

One could wonder whether the monotonicity of the variance, as was discovered in \cite{AM} in a reflected L\'evy context, would carry over to any stochastically monotone Markov process satisfying Condition~\ref{cond:1}. However, as it turns out, this particular result from \cite{AM} essentially follows due to the specific properties of reflected L\'evy processes (or, in the discrete-time case, reflected random walks) and is not true in general. One elementary counterexample is the following. Let $\{N_t\,|\,t\ge 0\}$ be a Poisson process (starting at $0$) and take $X_t=(k+N_t)\wedge m$. Then $\{X_t\,|\,t\ge 0\}$ is a Markov process with state space $\{i\,|\,i\le m\}$, with an initial value $k$ and an absorbing barrier $m$. On $k\le m$ we have that $(k+N_t)\wedge m$ is nondecreasing in $k$ and $X_t-X_0=N_t\wedge(m-k)$  is nonincreasing in $k$ and thus this is a stochastically monotone Markov process satisfying Condition~\ref{cond:1}. Clearly, $\text{Var}(X_0)=0$ and since $X_t\to m$ (a.s.) as $t\to \infty$ then by bounded convergence ($k\le X_t\le m$) the variance vanishes as $t\to\infty$. Since the variance is strictly positive for all $0<t<\infty$, it cannot be monotone in $t$.

{
}
\end{document}